\documentclass [a4paper, 12pt, reqno]{amsart}
\usepackage [latin1]{inputenc}
\usepackage {a4}
\usepackage{amscd}
\usepackage{epsfig}
\usepackage{amssymb}
\usepackage{amsmath}
\usepackage{amsthm}
\usepackage[T1]{fontenc}
\usepackage{ae,aecompl}
\usepackage[arrow, matrix, curve]{xy}

\usepackage{color}

\usepackage{geometry}
\newcommand{\R} {\ensuremath{\mathbb{R}}}

\newcommand{\C} {\ensuremath{\mathbb{C}}}

\newcommand{\OO}{\mathcal{O}}
\renewcommand{\o}[1]{\overline{#1}}

\newcommand{\dq}{\overline{\partial}}
\newcommand{\wt}[1]{\widetilde{#1}}

\DeclareMathOperator{\Sing}{Sing}

\DeclareMathOperator{\Dom}{Dom}

\newtheorem {satz} {Satz} [section]

\newtheorem {thm} [satz] {Theorem}

\DeclareMathOperator{\supp}{supp}

\renewcommand{\theta}{\vartheta}

\title[Canonical sheaves at isolated canonical Gorenstein singularities] 
{Canonical sheaves at isolated canonical Gorenstein singularities}

\dedicatory{Dedicated to the memory of Nils {\O}vrelid}

\author{J. Ruppenthal}

\address{Department of Mathematics, University of Wuppertal, Gau{\ss}str. 20, 42119 Wuppertal, Germany.}
\email{ruppenthal@uni-wuppertal.de}

\date{\today}



\begin{document}

\begin{abstract} 
It is well known that
the Grauert--Riemenschneider canonical sheaf $\mathcal{K}_X$ of holomorphic square-integrable $n$-forms
is a central tool in $L^2$-theory for the $\dq$-operator on a singular complex space $X$ of pure dimension $n$.
It was shown a few years ago that a comprehensive $L^2$-theory requires also 
the study of the sheaf $\mathcal{K}_X^s$ of holomorphic square-integrable $n$-forms with a Dirichlet boundary condition
at the singular set of $X$. In the present paper, we describe and classify the behaviour of $\mathcal{K}_X^s$
in isolated canonical Gorenstein singularities,
and give applications to the $L^2$-theory for the $\dq$-operator on such spaces.
\end{abstract}

\maketitle

~\\[-16mm]
\section{Introduction}




The canonical sheaf is one of the most important objects associated to a smooth complex manifold $X$.
It plays a crucial role e.g. in Serre duality, the Kodaira vanishing theorem or in classification theory (Kodaira dimension).
Unfortunately, these roles do not generalise in the same way to the singular setting.
On a singular complex space $X$ of pure dimension $n$,
there are different ways to define sheaves of holomorphic $n$-forms.

The two most important and well-studied options are
the Grothendieck dualising sheaf $\omega_X$, which allows for the generalisation of Serre duality,
and the Grauert--Riemenschneider canonical sheaf $\mathcal{K}_X$, which allows for the generalisation of the Kodaira
vanishing theorem (in this setting also called Grauert--Riemenschneider vanishing theorem).
The two notions do not coincide in general (see \cite{GR}).

On a normal complex space $X$ of pure dimension $n$, the dualising sheaf $\omega_X$ can be simply realised as the sheaf of 
holomorphic $n$-forms on the regular locus of the variety (see, e.g., \cite{GR}, Satz 3.1), 
whereas the Grauert--Riemenschneider canonical sheaf $\mathcal{K}_X$
is the sheaf of holomorphic square-integrable $n$-forms on the regular locus of the variety. So, we have an inclusion
\begin{eqnarray}\label{eq:intro001}
\mathcal{K}_X &\subset& \omega_X.
\end{eqnarray}
Moreover, $\mathcal{K}_X = \pi_* \mathcal{K}_M$, where $\pi: M \rightarrow X$ is any resolution of singularities (see \cite{GR}),
as the $L^2$-property of $n$-forms is invariant under modifications.

Let now, for the moment, $X$ be a Gorenstein space (i.e., normal, Cohen-Macaulay and $\omega_X$ locally free),
for example a normal locally complete intersection. Then $\omega_X=\OO(K_X)$, where $K_X$ is the canonical line bundle of $X$,
and the behavior of $\omega_X$ under a resolution of singularities $\pi: M\rightarrow X$
determines the type of the singularities of $X$ in the sense of the Minimal Model Program.
The inclusion \eqref{eq:intro001} becomes an equality if and only if $X$ has only canonical singularities
(see, e.g., \cite{Duke}, Theorem 4.3),
and then, both, duality and Grauert--Riemenschneider vanishing hold for $\mathcal{K}_X = \omega_X$.

\bigskip
In \cite{Duke}, we introduced another, new kind of canonical sheaf on an arbitrary Hermitian complex space $X$ of pure dimension $n$,
which we denote by $\mathcal{K}_X^s$. 
It is the sheaf of germs of holomorphic square-integrable
$n$-forms which satisfy a kind of Dirichlet boundary condition at the singular set of $X$.
It is defined as the kernel of the $\dq_s$-operator on square-integrable $(n,0)$-forms,
where the $\dq_s$-operator is a localised version of the $L^2$-closure of the $\dq$-operator
acting on forms with support away from the singular set (see Section \ref{sec:dq-complexes}).
The $\dq_s$-closed functions are precisely the weakly holomorphic functions,
and the $\dq_s$-complex of $(0,q)$-forms is exact in a singular point if and only if the singularity is rational
(\cite{Serre}, Theorem 1.1 and its corollaries).

If $X$ is a Hermitian complex space with only isolated singularities,
it is shown in \cite{Duke}, Theorem 1.9, that the $\dq_s$-complex
\begin{eqnarray}\label{eq:intro1}
0\rightarrow \mathcal{K}_X^s \hookrightarrow \mathcal{F}^{n,0} \overset{\dq_s}{\longrightarrow}
\mathcal{F}^{n,1} \overset{\dq_s}{\longrightarrow} \mathcal{F}^{n,2} \overset{\dq_s}{\longrightarrow} ... 
\end{eqnarray}
is a fine resolution of $\mathcal{K}_X^s$, where the $\mathcal{F}^{n,q}$ are the sheaves of germs of $L^2$-forms
in the domain of the $\dq_s$-operator.
It follows that
\begin{eqnarray*}
H^q(X, \mathcal{K}_X^s) &\cong& H^{n,q}_{s,loc}(X) := H^q\big( \Gamma(X,\mathcal{F}^{n,*})\big),\\
H^q_{cpt}(X, \mathcal{K}_X^s) &\cong& H^{n,q}_{s,cpt}(X) := H^q\big( \Gamma_{cpt}(X,\mathcal{F}^{n,*})\big),
\end{eqnarray*}
i.e. the cohomology of $\mathcal{K}_X^s$ can be represented by the $L^2$-$\dq_s$-cohomology.







\bigskip

In \cite{OV}, \cite{Duke} and \cite{Serre}, the new canonical sheaf $\mathcal{K}_X^s$ was the key tool
to understand the $L^2$-cohomology for the $\dq$-operator in the sense of distributions
at {\it isolated} singularities in terms of a resolution of singularities.
It was used that $\mathcal{K}_X^s$ is a {\it coherent} analytic sheaf and that there exists a resolution of singularities
$\pi: M \rightarrow X$ such that
\begin{eqnarray}\label{eq:divisor}
\mathcal{K}_X^s &=& \pi_* \big( \mathcal{K}_M \otimes \OO(-D)\big) ,
\end{eqnarray}
where $D$ is a certain effective divisor on $M$ (\cite{Duke}, Theorem 1.10). The proof of the coherence
of $\mathcal{K}_X^s$ in \cite{Duke} was rather involved, based on some sophisticated
arguments from \cite{OV}, and is so far only known for isolated singularities.
The first objective of the present paper is to give a more direct, concise and simple proof
of this issue:

\begin{thm}\label{thm:main00}
Let $X$ be a complex space of pure dimension with only isolated \linebreak
singularities. Then $\mathcal{K}_X^s$ is coherent.
\end{thm}

The argument is as follows: $\mathcal{K}_X^s$ is an analytic subsheaf of the coherent sheaf $\mathcal{K}_X$.
So, we only have to show that $\mathcal{K}_X^s$ is finitely generated, and this is clear for regular points.
But, if $x_0\in X \subset \C^N$ is a singular point in $X$, where $X$ is locally embedded in some $\C^N$,
then a standard basis of holomorphic $n$-forms in $\C^N$, restricted to $X$, is $\dq_s$-closed
and generates $(\mathcal{K}_X^s)_w=(\mathcal{K}_X)_w$ in regular points $w\in X$ closeby.
This is elaborated in Section \ref{sec:coherence}.

\medskip
It remains to obtain a better understanding of $\mathcal{K}_X^s$.
Our second main result in the present paper is as follows:

\begin{thm}\label{thm:main01}
Let $X$ be a singular complex space and $x_0\in X$ an isolated canonical Gorenstein singularity.
Then, on a small neighborhood of $x_0$, we have either
\begin{eqnarray}\label{eq:type0}
\mathcal{K}_X^s \ =\  \mathcal{K}_X \ = \ \omega_X,
\end{eqnarray}
or
\begin{eqnarray}\label{eq:type1}
\mathcal{K}_X^s \ =\ \mathfrak{m} \otimes \mathcal{K}_X \ =\  \mathfrak{m} \otimes \omega_X,
\end{eqnarray}
where $\mathfrak{m}$ is the maximal ideal sheaf of $\OO_X$ at $x_0$.

Let $\pi: M \rightarrow X$ be a resolution of singularities such that the exceptional set is a hypersurface, 
and let $Z:=\pi^{-1}\big(\{x_0\}\big)$
be the unreduced exceptional divisor. In the first case, $\mathcal{K}_X^s=\pi_* \mathcal{K}_M$,
whereas in the second case, we have
\begin{eqnarray}\label{eq:type1b}
\mathcal{K}_X^s  &=&   \pi_* \big( \mathcal{K}_M \otimes \OO(-Z)\big) .
\end{eqnarray}
\end{thm}


\bigskip
Note that this characterisation applies to a wide range of interesting isolated singularities.
First, locally complete intersections are Gorenstein. Second, the singularities that
appear in the search for minimal models are canonical.

Anyway, Theorem \ref{thm:main01} is not only interesting in its own right,
but has strong implications for the $L^2$-theory of the $\dq$-operator,
which refine the results from \cite{OV} and \cite{Duke}
considerably.
Let $\dq_w$ be the $\dq$-operator in the sense of distributions
on $L^2$-forms on the regular part of the variety (see Section \ref{sec:dq-complexes}).
The sheaves of germs of $(p,q)$-forms in the domain of $\dq_w$ are denoted by $\mathcal{C}^{p,q}$.
Then, the $\dq_w$-complex
\begin{eqnarray}\label{eq:intro2}
0\rightarrow \mathcal{O}_X \hookrightarrow \mathcal{C}^{0,0} \overset{\dq_w}{\longrightarrow}
\mathcal{C}^{0,1} \overset{\dq_w}{\longrightarrow} \mathcal{C}^{0,2} \overset{\dq_w}{\longrightarrow} ... 
\end{eqnarray}
is by no means exact in general (see, e.g., \cite{Rp7}, Theorem 1.2).
This will be used in the proof of Theorem \ref{thm:main04} below.

Let
\begin{eqnarray*}
 H^{0,q}_{w,loc}(X) &:=& H^q\big( \Gamma(X,\mathcal{C}^{0,*})\big)
\end{eqnarray*}
denote the $L^2$-$\dq_w$-cohomology.


\begin{thm}\label{thm:main02}
Let $X$ be a Hermitian complex space with only isolated canonical Gorenstein singularities,
$\pi: M\rightarrow X$ a resolution of singularities, and $\Omega\subset X$ an open set. 

If all the singularities in $\Omega$ are of the first kind \eqref{eq:type0} in Theorem \ref{thm:main01},
then 
pull-back of forms induces for any $q\geq0$ a natural isomorphism
\begin{eqnarray}\label{eq:main02a}
H^{n,q}_{s,loc}(\Omega)
\overset{\cong}{\longrightarrow} H^q\big(\pi^{-1}(\Omega),\mathcal{K}_M\big),
\end{eqnarray}
and push-forward of forms induces for any $q\geq0$ a natural isomorphism
\begin{eqnarray}\label{eq:main02b}
H^q\big(\pi^{-1}(\Omega),\OO_M\big) 
\overset{\cong}{\longrightarrow} H^{0,q}_{w,loc}(\Omega).
\end{eqnarray}
As canonical singularities are rational, it follows particularly that the $\dq_w$-$L^2$-complex \eqref{eq:intro2}
is exact in singularities of the first kind \eqref{eq:type0}.

If all the singularities in $\Omega$ are of the second kind \eqref{eq:type1} in Theorem \ref{thm:main01},
then there exist for any $q\geq0$
a natural isomorphism
\begin{eqnarray}\label{eq:main02c}
H^{n,q}_{s,loc}(\Omega)
\overset{\cong}{\longrightarrow} H^q\big(\pi^{-1}(\Omega),\mathcal{K}_M\otimes\mathcal{O}_M(-Z) \big),
\end{eqnarray}
and 
\begin{eqnarray}\label{eq:main02d}
H^q\big(\pi^{-1}(\Omega),\OO_M(Z)\big) 
\overset{\cong}{\longrightarrow} H^{0,q}_{w,loc}(\Omega),
\end{eqnarray}
where $Z=\pi^{-1}(\Sing X)$ is the unreduced exceptional divisor.
\end{thm}

Note that the cohomology groups $H^{n,q}_{w,loc}(\Omega)$ and $H^{0,q}_{s,loc}(\Omega)$, respectively,
have been expressed in terms of a resolution of singularities $\pi: M \rightarrow X$
already in \cite{Duke} and \cite{Serre}, namely
\begin{eqnarray*}
H^{n,q}_{w,loc}(\Omega) & \cong & H^q\big(\pi^{-1}(\Omega),\mathcal{K}_M\big) ,\\
H^{0,q}_{s,loc}(\Omega) &\cong & H^q\big(\pi^{-1}(\Omega),\OO_M\big) ,
\end{eqnarray*}
where we assume that $\Omega$ is holomorphically convex for the second statement (see \cite{Serre}, Theorem 1.1),
whereas the first holds (analogously to the statement of Theorem \ref{thm:main02}) even without that condition.

Hence, Theorem \ref{thm:main02} completes our understanding of the $L^2$-cohomology for the $\dq$-operator
on the level of $(0,q)$- and $(n,q)$-forms.

\bigskip
As a second application of Theorem \ref{thm:main01},
we obtain the following vanishing theorem of Kodaira--Grauert--Riemenschneider--Takegoshi type:

\begin{thm}\label{thm:main03}
Let $N$ be a compact K\"ahler manifold.
Let $X \subset N$ be an irreducible subvariety with only isolated canonical Gorenstein singularities
of the first kind (in Theorem \ref{thm:main01}),
$L \rightarrow X$ an almost positive holomorphic line bundle, and let $q>0$.
Then
\begin{eqnarray*}
H^q\big(X, \mathcal{K}_X^s (L) \big) \ = \ H^{n,q}_s(X,L) \ = \ H^{0,n-q}_w(X,L^*) \ = \ 0.
\end{eqnarray*}

\end{thm}

Here, $\mathcal{K}_X^s(L)$ denotes the sheaf of $\dq_s$-closed $(n,0)$-forms with values in $L$,
$H^{n,q}_s(X,L)$ is the global cohomology of the $\dq_s$-$L^2$-complex with values in $L$,
and $H^{0,n-q}_w(X,L^*)$ the global cohomology of the $\dq_w$-$L^2$-complex with values in the
dual bundle $L^*$.

Recall that a holomorphic line bundle $L$ on an irreducible (compact) complex space is called {\it almost positive}
if there exists a Hermitian metric on $L$ whose curvature is semipositive everywhere and positive on some open set.
So, Theorem \ref{thm:main03} applies particularly to positive line bundles.

\bigskip

We conclude this paper by characterising canonical singularities in dimension two (they are always Gorenstein)
with respect to the distinction in Theorem \ref{thm:main01}.
In dimension two, i.e., on a complex surface, a singularity is canonical if and only if it is a rational double point.
Such points
are well-studied and have been classified a long time ago as the so-called du~Val singularities,
see, e.g.,  the survey \cite{Du}. The possible singularities are of type $A_n$, $n\geq 1$, $D_n$, $n\geq 4$, 
or of exceptional type $E_6$, $E_7$ or $E_8$, and can be realized as isolated hypersurface singularities in $\C^3$.

\begin{thm}\label{thm:main04}
Let $x_0$ be a canonical surface singularity.
If $x_0$ is an $A_n$-singularity, then $x_0$ is of the first kind \eqref{eq:type0} in Theorem \ref{thm:main01}.
Otherwise, i.e., if $x_0$ is of type $D_n$, $n\geq 4$, $E_6$, $E_7$ or $E_8$, then
$x_0$ is of the second kind \eqref{eq:type1} in Theorem \ref{thm:main01}.
\end{thm}

This shows particularly that both kinds of singularities do exist, actually already in dimension two.
To prove that the $A_n$-singularities are of the first kind, we will compute explicitly
that the generator of $\mathcal{K}_X$, the so-called {\it structure form} or {\it Poincar\'e residue},
is in the domain of $\dq_s$. This shows that $\mathcal{K}_X \subset \mathcal{K}_X^s$,
and the other inclusion is clear by definition.

For the proof of the fact that the other canonical surface singularities are of the second kind,
we make use of the observation that the $\dq_w$-$L^2$-complex \eqref{eq:intro2}
is exact in singularities of the first kind (see Theorem \ref{thm:main02}).
But \eqref{eq:intro2} is actually not exact at $\mathcal{C}^{0,1}$ 
in singularities of type $D_n$, $n\geq 4$, $E_6$, $E_7$ and $E_8$
as we have shown in \cite{Slarw}, Theorem 5.6.
The fact that there exist obstructions to solving $\dq_w$ in $L^2$ at the $D_4$-singularity was
known already before, due to Pardon \cite{P1}, Proposition 4.13.

Let us mention an interesting fact in the context of Theorem \ref{thm:main04}.
In the case of surface singularities, \eqref{eq:divisor} can be replaced by the more precise
observation that
\begin{eqnarray*}
\mathcal{K}_X^s &=& \pi_* \big( \mathcal{K}_M \otimes \OO(|Z|-Z)\big) ,
\end{eqnarray*}
where $Z$ is the unreduced exceptional divisor (see \cite{Duke}, Theorem 1.13). 
The $A_n$-singularities actually have a resolution with so-called reduced fundamental cycle,
$Z=|Z|$. This would give another proof of the fact that these singularities are of the first kind.
On the other hand, for the singularities of type $D_n$, $n\geq 4$, $E_6$, $E_7$, and $E_8$,
we see that there can not exist a resolution with reduced fundamental cycle,
and we deduce
\begin{eqnarray*}
\mathcal{K}_X^s &=& \pi_* \big( \mathcal{K}_M \otimes \OO(|Z|-Z)\big) \ \ = \ \  \pi_* \big( \mathcal{K}_M \otimes \OO(-Z)\big).
\end{eqnarray*}

\bigskip
The present paper is organized as follows. In Section \ref{sec:dq-complexes},
we recall the definition of the $\dq_w$- and the $\dq_s$-operator, and show that
bounded forms in the domain of $\dq_w$ are also in the domain of $\dq_s$.
Section \ref{sec:coherence} is devoted to the concise proof of the fact that $\mathcal{K}_X^s$
is coherent, and we prove Theorem \ref{thm:main01} in Section \ref{thm:cGs}.
In Section \ref{sec:smooth}, we show how the $L^2$-cohomology of $X$ can be represented by a smooth
model in terms of a resolution of singularities.
Section \ref{sec:vanishing} contains the proof of Theorem \ref{thm:main03}.
Finally, in the last section, we characterise rational double points in terms of $\mathcal{K}_X^s$.

\newpage

\section{Two $\dq$-complexes on singular complex spaces}\label{sec:dq-complexes}

Let us recall shortly some essential notation from \cite{Duke}, \cite{Serre} and \cite{Toulouse},
which is used analogously in \cite{Slarw}, and which appears implicitly also in \cite{OV}.
Let $(X,h)$ always be a (singular) Hermitian complex space of pure dimension $n$.
A Hermitian complex space $(X,h)$ is a reduced complex space $X$ with a metric $h$ on the regular part such that the following holds:
If $x\in X$ is an arbitrary point, then there exists a neighborhood $U=U(x)$ and a
biholomorphic embedding of $U$ into a domain $G$ in $\C^N$ and an ordinary smooth Hermitian metric in $G$
whose restriction to $U$ is $h|_U$.

\subsection{The $\dq_w$-$L^2$-complex}

It is pretty fruitful to consider the following concept of locally square-integrable forms on $X$.
For any open set $U\subset X$ we define
\begin{eqnarray*}
L_{loc}^{p,q}(U):=\{f \in L_{loc}^{p,q}(U-\Sing X): f|_K \in L^{p,q}_h(K-\Sing X)\ \forall K\subset\subset U\},
\end{eqnarray*}
where $L^{p,q}_h(K- \Sing X)$ are the square-integrable $(p,q)$-forms with respect to the metric $h$
on the smooth manifold $K-\Sing X$. Now
$$\mathcal{L}^{p,q}(U) := L_{loc}^{p,q}(U)$$
defines the sheaf $\mathcal{L}^{p,q}\rightarrow X$ of (locally) square-integrable $(p,q)$-forms on $X$.

On $L_{loc}^{p,q}(U)$, we denote by
$$\dq_w: L_{loc}^{p,q}(U) \rightarrow L_{loc}^{p,q+1}(U)$$
the $\dq$-operator in the sense of distributions on $U-\Sing X$ which is closed and densely defined.
The subscript refers to $\dq_w$ as an operator in a weak sense.
By
$$\mathcal{C}^{p,q}(U) := \mathcal{L}^{p,q}(U) \cap\Dom\dq_w$$
we obtain the sheaves of $L^2$-forms in the domain of $\dq_w$.
It is easy to see that these admit partitions of unity,
and so we obtain for any $0\leq p \leq n$ a fine sequence
\begin{eqnarray}\label{eq:Cseq1}
\mathcal{C}^{p,0} \overset{\dq_w}{\longrightarrow} \mathcal{C}^{p,1} \overset{\dq_w}{\longrightarrow} \mathcal{C}^{p,2} \overset{\dq_w}{\longrightarrow} ...
\end{eqnarray}
Of particular importance is the case $p=n=\dim X$, in which
\begin{eqnarray*}
\mathcal{K}_X := \ker \dq_w \subset \mathcal{C}^{n,0}
\end{eqnarray*}
is just the canonical sheaf of Grauert and Riemenschneider introduced in \cite{GR},
since the $L^2$-property of $(n,0)$-forms remains invariant under modifications (cf. \cite{Duke}).

The $L^{2,loc}$-Dolbeault cohomology with respect to the $\dq_w$-operator on an open set $U\subset X$
is the cohomology of the complex \eqref{eq:Cseq1}, which is denoted by
\begin{eqnarray*}
H^{p,q}_{w,loc}(U) &:=& H^q\big(\Gamma(U,\mathcal{C}^{p,*})\big),
\end{eqnarray*}
and the cohomology with compact support is
\begin{eqnarray*}
H^{p,q}_{w,cpt}(U) &:=& H^q\big(\Gamma_{cpt}(U,\mathcal{C}^{p,*})\big).
\end{eqnarray*}
We will use the analogous notation, i.e., $\big(\mathcal{C}_M^{p,*}, \dq_w\big)$, to denote the usual
$\dq_w$-complex of $L^2$-forms in the domain of $\dq_w$ on a complex manifold $M$
(this follows with the same definitions, the singular set is simply empty).

\bigskip

\subsection{The $\dq_s$-$L^2$-complex}

In \cite{Duke}, we also introduced a suitable local realization of a minimal closed extension of the $\dq$-operator.
This is the $\dq$-operator with a kind of Dirichlet boundary condition at the singular set $\Sing X$ of $X$.
Let
$$\dq_s(U): L_{loc}^{p,q}(U) \rightarrow L_{loc}^{p,q+1}(U)$$
be defined as follows. We say that $f\in\Dom\dq_w$ is in the domain of $\dq_s$ if there exists a
sequence of forms $\{f_j\}_j \subset \Dom\dq_w \subset L_{loc}^{p,q}(U)$ with essential support away from the singular set in $U$,
$$\supp f_j \cap \Sing U = \emptyset,$$
such that
\begin{eqnarray}\label{eq:ds1}
f_j \rightarrow f &\mbox{ in }& L^{p,q}_h(K-\Sing X),\\
\dq_w f_j \rightarrow \dq_w f &\mbox{ in }& L^{p,q+1}_h(K-\Sing X)\label{eq:ds2}
\end{eqnarray}
for any compact subset $K\subset\subset U$. The condition is empty if there are no singularities in $U$.
The subscript refers to $\dq_s$ as an extension in a strong sense.
It is closed and densely defined.
Note that we can assume without loss of generality
(by use of a partition of unity and smoothing with Dirac sequences) that the forms $f_j$ are smooth 
on relatively compact subsets of $U$.
The reason is that the $f_j$ have support away from the singular set.

We obtain now by
$$\mathcal{F}^{p,q}(U) := \mathcal{L}^{p,q}(U) \cap \Dom \dq_s(U).$$
the sheaves of $L^{2,loc}$-forms in the domain of $\dq_s$.
As for $\mathcal{C}^{p,q}$, it is again not hard to see that the sheaves $\mathcal{F}^{p,q}$ are fine,
and so there is for any $0\leq p \leq n$ a fine complex
\begin{eqnarray}\label{eq:Fseq1}
\mathcal{F}^{p,0} \overset{\dq_s}{\longrightarrow} \mathcal{F}^{p,1} \overset{\dq_s}{\longrightarrow} \mathcal{F}^{p,2} \overset{\dq_s}{\longrightarrow} ...
\end{eqnarray}
We can now introduce the sheaf
\begin{eqnarray}\label{eq:KXs}
\mathcal{K}^s_X := \ker \dq_s \subset \mathcal{F}^{n,0}
\end{eqnarray}
which we may call the {\it canonical sheaf of holomorphic $n$-forms with Dirichlet boundary condition}.
The main objective of the present paper is to study $\mathcal{K}_X^s$ at isolated singularities. 
We will in particular develop nice representations in terms of a resolution of singularities
for $\mathcal{K}_X^s$ and its cohomology.

The $L^{2,loc}$-Dolbeault cohomology with respect to the $\dq_s$-operator on an open subset $U\subset X$
is the cohomology of the complex \eqref{eq:Fseq1} which is denoted by
\begin{eqnarray*}
H^{p,q}_{s,loc}(U) &:=& H^q\big(\Gamma(U,\mathcal{F}^{p,*})\big),
\end{eqnarray*}
and the cohomology with compact support is
\begin{eqnarray*}
H^{p,q}_{s,cpt}(U) &:=& H^q\big(\Gamma_{cpt}(U,\mathcal{F}^{p,*})\big).
\end{eqnarray*}

\subsection{On the domain of the $\dq_s$-operator}\label{ssec:domaindqs}

In general, it is difficult to decide whether a differential form is in the domain of the $\dq_s$-operator or not.
However, we have at least the following very useful criterion:

\begin{thm}\label{thm:bounded}
Locally bounded forms in the domain of the $\dq_w$-operator are also in the domain of the $\dq_s$-operator.
\end{thm}

Here, locally bounded means bounded on $K^*=K-\Sing X$ for compact sets $K\subset\subset X$.
A proof of Theorem \ref{thm:bounded}, using resolution of singularities,
was given in \cite{parabolic}, Theorem 1.6.
But Theorem \ref{thm:bounded} can be achieved also in a more direct way
as an easy consequence of Sibony's results on extension of analytic objects \cite{Sibony}.
We recall it from \cite{Serre} for convenience of the reader.

\begin{proof}
As the problem is local by partition of unity, consider a Hermitian complex space $X$ of dimension $n$,
embedded locally in some open set $\Omega\subset \C^N$.
Let $T$ be the positive closed current of integration over $X$, and $A$ the singular set of $X$ (as a subset of $\Omega$),
which is a complete pluripolar set.
By \cite{Sibony}, Lemma 1.2, there exists a sequence $\{u_j\}$ of smooth plurisubharmonic functions, $0\leq u_j\leq 1$,
vanishing identically on a neighborhood of $A$, converging uniformly to $1_{\Omega\setminus A}$
on compact subsets of $\Omega\setminus A$.

Consider now compact sets $K, L \subset \Omega$ with $K\subset \mbox{int}(L)$,
and let $0\leq \chi\leq 1$ be a test function
with support in $L$ and $\chi =  1$ on $K$. Let $\omega:=i\partial\dq \|z\|^2$ be the K\"ahler form in $\C^N$. 
Then, similarly as in the proof of Theorem 1.3 in \cite{DiSi}, p. 361, we have
\begin{eqnarray*}
\| \dq u_j\|^2_{L^2(K\cap X)} &\leq& \int \chi^2 i\partial u_j\wedge \dq u_j \wedge T \wedge \omega^{n-1}
\leq \frac{1}{2} \int \chi^2 i\partial \dq u_j^2 \wedge T \wedge \omega^{n-1}\\
&=& \frac{1}{2} \int \chi^2 i\partial \dq (u_j^2-1) \wedge T \wedge \omega^{n-1}
= \frac{1}{2} \int (u_j^2 -1) i\partial\dq \chi^2 \wedge T \wedge \omega^{n-1}\\
&\leq& C_\chi \|u_j^2-1\|_{L^1(L\cap X)} \overset{j\rightarrow \infty}{\longrightarrow} 0,
\end{eqnarray*}
where we have convergence to $0$ because the $u_j^2$ converge uniformly to $1$ on compact subsets of $\Omega\setminus A$.

So, let $f$ be a bounded form in the domain of $\dq_w$.
Then it is easy to see (using the H\"older inequality)
that $f_j:= u_j f$ is a sequence as required in \eqref{eq:ds1}, \eqref{eq:ds2}.
\end{proof}


\bigskip

\section{Coherence of $\mathcal{K}_X^s$ at isolated singularities}\label{sec:coherence}

%
%

In this section, we give a new, short proof of Theorem \ref{thm:main00}.

\medskip
First, note that the Grauert--Riemenschneider canonical sheaf
$\mathcal{K}_X$ is a coherent analytic sheaf by Grauert's direct image theorem,
as $\mathcal{K}_X= \pi_* \mathcal{K}_M$ for any resolution of singularities $\pi: M \rightarrow X$ (see \cite{GR}).

\medskip
Now then, $\mathcal{K}_X^s \subset \mathcal{K}_X$ is an analytic subsheaf, as is easy to see:
If $f$ is a section in $\mathcal{K}_X^s(U)$ such that \eqref{eq:ds1}, \eqref{eq:ds2} holds,
and $h$ is a holomorphic function on the open set $U$, then
\begin{eqnarray*}
h\cdot f_j \ \rightarrow \ h\cdot f &\mbox{ in }& L^{p,q}_h(K-\Sing X),\\
\dq_w (h \cdot f_j) = h\cdot \dq_w f_j \ \rightarrow \ h\cdot \dq_w f=0 &\mbox{ in }& L^{p,q+1}_h(K-\Sing X)
\end{eqnarray*}
on any compact subset $K\subset\subset U$. Hence, $\dq_s(h\cdot f)=0$, too.
A very similar argument shows, by the way, that the sheaves $\mathcal{C}^{p,q}$ and $\mathcal{F}^{p,q}$ are fine.

\medskip
It only remains to show that $\mathcal{K}_X^s$ is finitely generated. So,
assume that $x_0 \in X$ is an isolated singularity and that an open neighborhood $V$ of $x_0$ is embedded
holomorphically in some $\C^N$.
The ideal $\mathcal{K}_{X,x_0}^s$ in $\mathcal{K}_{X,x_0}$ itself is clearly finitely generated,
but we need to find finitely many sections of $\mathcal{K}_{X,x_0}^s$ which generate $\mathcal{K}_X^s$ in a neighborhood of $x_0$.
But this is achieved by considering a standard basis of holomorphic $n$-forms
$$dz_{j_1} \wedge dz_{j_2}\wedge ... \wedge dz_{j_n}$$
in $\C^N$. These forms are bounded and clearly $\dq_w$-closed on the regular locus of $X$.
So, by Theorem \ref{thm:bounded}, restriction to $X$ gives sections in $\mathcal{K}_X^s$,
and the full set of all $\binom{N}{n}$ such forms generates $\mathcal{K}_X^s$ in any
regular point of $X$ by standard differential geometry.



\newpage

\section{$K_X^s$ in canonical Gorenstein singularities}\label{thm:cGs}

In this section, we prove Theorem \ref{thm:main01}.

\medskip

Let $x_0 \in X$ be an isolated canonical Gorenstein singularity.
We may assume that a small neighborhood $U=U(x_0)$ of $x_0$ is embedded holomorphically into some open domain in $\C^N$
such that the Hermitian metric on $X$ is just the restriction of a regular Hermitian metric on $\C^N$,
and that $U$ contains no further singularities. As we consider a local problem, we may simply say that $X=U$.

So, $X$ is Gorenstein, i.e., normal and Cohen-Macaulay and the Grothendieck dualising sheaf $\omega_X$ is locally
free of rank one:
\begin{eqnarray*}
\omega_X &\cong& \OO_X( K_X ),
\end{eqnarray*}
where $K_X$ is the canonical divisor of $X$. By restricting $X$ further if necessary,
we can assume that $K_X$ is trivial so that $\omega_X \cong \OO_X$.

Let $\iota: X  - \Sing X \hookrightarrow X$ be the embedding of the regular locus.
By normality,
$$\omega_X \ = \ \iota_*( \omega_{X - \Sing X}),$$
i.e., sections of $\omega_X$ are nothing but holomorphic $n$-forms on the regular locus (see \cite{GR}, Satz 3.1).
Hence, there is the trivial inclusion of the Grauert--Riemenschneider sheaf of square-integrable holomorphic
$n$-forms on the regular set into the dualizing sheaf:
\begin{eqnarray}\label{eq:KX001}
\mathcal{K}_X &\subset& \omega_X.
\end{eqnarray}

This inclusion becomes an equality if and only if the singularities of $X$ are canonical
(see e.g. \cite{Duke}, Theorem 4.3). The reason is as follows: $X$ has canonical singularities
by definition in terms of discrepancies if and only if $\pi^* \omega_X \subset \mathcal{K}_M$
for any resolution of singularities $\pi: M \longrightarrow X$. If \eqref{eq:KX001} holds,
this is equivalent to $\mathcal{K}_X = \omega_X$ by some easy computations
using the natural inclusions $\pi^* \pi_* \mathcal{G} \subset \mathcal{G}$ and $\mathcal{H} \subset \pi_* \pi^* \mathcal{H}$
for analytic sheaves $\mathcal{G}$, $\mathcal{H}$.

Summing up, we have seen that
\begin{eqnarray}\label{eq:canGor}
\mathcal{K}_X & = & \omega_X \ \ \cong \ \ \OO_X
\end{eqnarray}
at canonical Gorenstein singularities. Hence, we can consider $\mathcal{K}_X^s$ as a sheaf of ideals in
$\OO_X$. Let $\mathfrak{m}$ be the maximal ideal sheaf at $x_0$. In the next section, we will show that
\begin{eqnarray}\label{eq:thm01}
\mathfrak{m} \otimes \mathcal{K}_X & \subset & \mathcal{K}_X^s.
\end{eqnarray}
In view of \eqref{eq:canGor}, this means nothing else but that $\mathcal{K}_X^s$ is a sheaf of ideals in the structure sheaf
that contains the maximal ideal sheaf at $x_0$. Hence, it is either isomorphic to the structure sheaf itself,
or to the maximal ideal sheaf. That proves the first part of Theorem \ref{thm:main01}
as soon as \eqref{eq:thm01} is shown. This will be done in the next subsection.

\smallskip
\subsection{Cut-off functions}\label{ssec:cut-off}

Assume without loss of generality that
$$x_0=0 \in X \subset \C^N.$$
We will use the following cut-off functions to approximate forms
by forms with support away from the isolated singularity. 
As in the proof of Proposition~3.3 in \cite{PS1}, from where we have adopted this construction,
let $\rho_k: \R\rightarrow [0,1]$, $k\geq 1$, be smooth decreasing cut-off functions
satisfying
$$\rho_k(x)=\left\{\begin{array}{ll}
1,\  x\leq k,\\
0,\  x\geq k+1,
\end{array}\right.$$
and $|\rho_k'|\leq 2$. Moreover, let $r: \R_+\rightarrow [0,1/2]$ be a smooth increasing function such that
$$r(x)=\left\{\begin{array}{ll}
x,\  & x\leq 1/4,\\
1/2,\  & x\geq 3/4,
\end{array}\right.$$
and $|r'|\leq 1$.
As cut-off functions we will use 
\begin{equation}\label{eq:cutoff1}
\mu_k(\zeta):=\rho_k\big(\log(-\log r(\|\zeta\|))\big)
\end{equation}
on $X$. Recall that $X$ has an isolated singularity at $x_0=0\in \C^N$.
Note that there is a
constant $C>0$ such that
\begin{equation}\label{eq:cutoff2}
\big| \dq \mu_k(\zeta)\big| \leq C\frac{\chi_k(|\zeta|)}{\|\zeta\| \big| \log\|\zeta\|\big|},
\end{equation}
where $\chi_k$ is the characteristic function of $[e^{-e^{k+1}}, e^{-e^k}]$.

\medskip
In order to prove \eqref{eq:thm01}, consider a germ
$$\eta \in  \big( \mathfrak{m} \otimes \mathcal{K}_X\big)_{x_0}.$$
We can assume that $\eta = f \cdot \omega$, where $f$ is a holomorphic function vanishing
in the origin and $\omega$ is a square-integrable holomorphic $n$-form on $X$.
As the $\mu_k$ are clearly uniformly bounded in $k$,
it follows by Lebesgue's dominated convergence that
\begin{eqnarray*}
\eta_k:= \mu_k \cdot \eta \ \longrightarrow \eta 
\end{eqnarray*}
in $L^{n,0}$. But \eqref{eq:cutoff2} shows that the $f \cdot \dq \mu_k$  are also uniformly bounded in $k$,
because $f$ is vanishing in the origin.
Hence, dominated convergence gives also
\begin{eqnarray*}
\dq \eta_k &=& f \cdot \dq \mu_k \wedge \omega \ \longrightarrow \ 0
\end{eqnarray*}
in $L^{n,1}$, as the forms $\dq \eta_k$ are uniformly square integrable and the support of $\dq \mu_k$ is vanishing.
Hence, $\eta$ is actually approximated in the graph norm by forms with support away from the singularity.
This means nothing else but $\eta \in \Dom\dq_s$ and shows $\eta
\in \big(\mathcal{K}_X^s\big)_{x_0}$, as desired.

\bigskip
\subsection{Resolution of singularities}\label{ssec:resolution}

Our next purpose is to find a smooth representing model for $\mathcal{K}_X^s$.
This will be done in terms of a resolution of singularities and finish the proof of Theorem \ref{thm:main01}. So, let
$$\pi: M \longrightarrow X$$
be a resolution of singularities. As seen above, we have to distinguish two cases.
If $\mathcal{K}_X^s$ equals the Grauert--Riemenschneider canonical sheaf,
then nothing remains to do as we obtain simply
$$\mathcal{K}_X^s \ \ = \ \ \mathcal{K}_X \ \ = \ \ \pi_* \mathcal{K}_M.$$
So, assume the second case,
$$\mathcal{K}_X^s \ \ = \ \ \mathfrak{m} \otimes \mathcal{K}_X.$$

\newpage
In order to obtain a decent statement, assume that the exceptional set $E$ of the resolution $\pi$
is a hypersurface.
As above, assume without loss of generality that
$$x_0=0 \in X \subset \C^N.$$

We will now consider the unreduced exceptional divisor $Z:=\pi^{-1}(\{x_0\})$.
Let $w_1, ..., w_N$ be the standard Euclidean coordinates in $\C^N$.
Then $Z$ is given as the common zero set of the holomorphic functions $\{\pi^*w_1, ..., \pi^*w_N\}$.
For another realization, let
$$F(w) \ \ := \ \ \left( \sum_{k=1}^N \|w_k\|^2\right)^{1/2}.$$
Then $\pi^*F$ vanishes precisely along the exceptional set $E$ of the resolution,
and the order of vanishing equals the order of $Z$.

We will now show that
\begin{eqnarray}\label{eq:nice01}
\pi_*\big( \mathcal{K}_M \otimes \OO(-Z)\big) &\subset& \mathcal{K}_X^s.
\end{eqnarray}
Note that, though we have $\pi_* \mathcal{K}_M = \mathcal{K}_X$ and $\pi_* \OO(-Z) = \mathfrak{m}$,
we cannot conclude $\pi_*\big( \mathcal{K}_M \otimes \OO(-Z)\big)=\mathcal{K}_X \otimes \mathfrak{m} = \mathcal{K}_X^s$
directly, as by no means $\pi_*(\mathcal{F}\otimes\mathcal{G}) = \pi_* \mathcal{F} \otimes \pi_*\mathcal{G}$
for arbitrary analytic sheaves $\mathcal{F}, \mathcal{G}$.

So, consider a small open neighborhood $U$ of $x_0=0$ in $X$ and a section
$$\omega \ \in \ \big( \mathcal{K}_M \otimes \OO(-Z)\big)(\pi^{-1}(U)).$$
So, $\omega$ is a holomorphic $n$-form, vanishing (at least) to the order of $Z$ on the exceptional set.
This means that $\omega/\pi^* F$ is still bounded, and particularly square-integrable.
Denote by $\pi_* \omega$ the push-forward of $\omega$, restricted to $M - E$,
to the regular part of $U$, which is here just $U^*:=U - \{0\}$.
We need to show that
\begin{eqnarray}\label{eq:nice02}
\pi_* \omega \ \in \ \mathcal{K}_X^s(U^*).
\end{eqnarray}
It is clear that $\pi_*\omega$ is a holomorphic $n$-form. It is again square-integrable
as this property is invariant under modifications for $n$-forms.
But we know even more. By the transformation formula,
\begin{eqnarray}\label{eq:nice03}
\|\pi_* \omega/ F\|^2_{L^2_X} &=& \int_{U^*} \frac{\pi_* \omega \wedge \overline{\pi_*\omega}}{F^2}
\ = \ \int_{\pi^{-1}(U^*)} \frac{\omega\wedge\o{\omega}}{\pi^* F^2} \ < \ \infty\ ,
\end{eqnarray}
as $\omega/\pi^*F$ is even bounded.
So, we conclude that even $\pi_* \omega/F$ is square-integrable on $X$,
and this is more than enough to prove \eqref{eq:nice02} in the following.

\bigskip
Let $\{\mu_k\}$ again be the cut-off sequence from the previous subsection,
defined in \eqref{eq:cutoff1}. So,
$$\omega_k \ := \ \mu_k \cdot \pi_* \omega$$
is a sequence of square-integrable $n$-forms with support away from the singularity.
The convergence to $\pi_* \omega$ in the graph norm follows as
in the previous subsection: $\omega_k \longrightarrow \pi_* \omega$ in $L^{n,0}$ for $k\rightarrow\infty$
by Lebesgue's dominated convergence. Moreover,
$|\dq \mu_k| \cdot \|\zeta\|$ is uniformly bounded by \eqref{eq:cutoff2} and 
$\pi_*\omega/\|\zeta\|$ is square-integrable as seen in \eqref{eq:nice03}.
Hence, the forms $\dq \mu_k \wedge \pi_* \omega$ are uniformly square-integrable,
independent of $k$. So,
\begin{eqnarray*}
\dq \omega_k &=&  \dq \mu_k \wedge \pi_* \omega \ \longrightarrow \ 0
\end{eqnarray*}
in $L^{n,1}$, as the support of $\dq \mu_k$ is vanishing.
This proves \eqref{eq:nice02} and so the desired \eqref{eq:nice01}.

\medskip
Consider now a germ
$$\eta \ \in \ \mathcal{K}_{X,x_0}^s \ = \ \big( \mathfrak{m} \otimes \mathcal{K}_X\big)_{x_0}.$$
As above, 
we can assume that $\eta = f \cdot \omega$, where $f$ is a holomorphic function vanishing
in the origin and $\omega$ is a square-integrable holomorphic $n$-form on $X$.
But then $\pi^* f$ is a section in $\OO(-Z)$, and $\pi^* \omega$ is a holomorphic $n$-form on $M$.
A priori, $\pi^*\omega$ is defined only outside the exceptional set, but it is square-integrable
by the transformation law, and so it extends holomorphically over the exceptional set.
Therefore, $\pi^* \eta=\pi^*f \cdot \pi^*\omega$ is a section of $\mathcal{K}_M \otimes \OO(-Z)$,
meaning that $\eta$ can be naturally considered as a section in $\pi_*( \mathcal{K}_M \otimes \OO(-Z))$.
This shows
$$\mathcal{K}_X^s = \mathfrak{m} \otimes \mathcal{K}_X 
\ \subset \ \pi_*\big( \mathcal{K}_M \otimes \OO(-Z) \big),$$
which, in combination with \eqref{eq:nice01}, concludes the proof of Theorem \ref{thm:main01}.

\bigskip
\section{Smooth representation of $\dq$-cohomology}\label{sec:smooth}

In this section, we prove Theorem \ref{thm:main02}:
We use a resolution of singularities to give a smooth representation of the $\dq_w$- and the $\dq_s$-cohomology
on the level of $(0,q)$ and $(n,q)$-forms.


\subsection{$L^2$-forms and resolution of singularities}\label{sssec:wresolution}

Let $\pi: M \rightarrow X$
be a resolution of singularities, i.e. a proper holomorphic surjection such that
\begin{eqnarray*}
\pi|_{M\setminus E}: M\setminus E \rightarrow X\setminus\Sing X
\end{eqnarray*}
is biholomorphic, where $E=|\pi^{-1}(\Sing X)|$ is the (reduced) exceptional set.
We may assume that $E$ is a divisor with only normal crossings,
i.e. the irreducible components of $E$ are regular and meet complex transversely (though we do not need that for the moment).
Let $Z:=\pi^{-1}(\Sing X)$ be the unreduced exceptional divisor.
Let
$$\gamma:= \pi^* h$$
be the pullback of the Hermitian metric $h$ of $X$ to $M$.
$\gamma$ is positive semidefinite (a pseudo-metric) with degeneracy locus $E$.
We give $M$ the structure of a Hermitian manifold with a freely chosen (positive definite)
metric $\sigma$. Then $\gamma \lesssim \sigma$ on compact subsets of $M$,
and $\gamma \sim \sigma$ on compact subsets of $M\setminus E$.
For an open set $U\subset M$, we denote by $L^{p,q}_{\gamma}(U)$ and $L^{p,q}_{\sigma}(U)$
the spaces of square-integrable $(p,q)$-forms with respect to the (pseudo-)metrics $\gamma$ and $\sigma$,
respectively. 

Since $\sigma$ is positive definite and $\gamma$ is positive semi-definite,
there exists a smooth function $g\in C^\infty(M,\R)$ such that
\begin{eqnarray}\label{eq:l2dV}
dV_\gamma = g^2 dV_\sigma.
\end{eqnarray}
This yields $|g| |\omega|_\gamma  = |\omega|_\sigma$
if $\omega$ is an $(n,0)$-form, and
$|\omega|_\sigma \lesssim_U |g||\omega|_\gamma$
on $U\subset\subset M$ if $\omega$ is a $(n,q)$-form, $0\leq q\leq n$.\footnote{
This statement implies particularly that $|\omega|_\sigma/|\omega|_\gamma$ is locally bounded on $M$ for $(n,q)$-forms.}
So, for an $(n,q)$ form $\omega$ on $U\subset\subset M$:
\begin{eqnarray}\label{eq:l2est2}
\int_U |\omega|_\sigma^2 dV_\sigma \lesssim_U \int_U g^{2} |\omega|_\gamma^2 g^{-2} dV_\gamma = \int_U |\omega|^2_\gamma dV_\gamma.
\end{eqnarray}
Conversely,
$|g| |\eta|_\gamma \lesssim_U |\eta|_\sigma$
on $U\subset\subset M$ if $\eta$ is a $(0,q)$-form, $0\leq q\leq n$.\footnote{
For $(0,q)$-forms, $|\omega|_\gamma/|\omega|_\sigma$ is locally bounded.}
So, for a $(0,q)$ form $\eta$ on $U\subset\subset M$:
\begin{eqnarray}\label{eq:l2est}
\int_U |\eta|_\gamma^2 dV_\gamma \lesssim_U \int_U g^{-2} |\eta|_\sigma^2 g^2 dV_\sigma = \int_U |\eta|^2_\sigma dV_\sigma.
\end{eqnarray}
For open sets $U\subset\subset M$ and all $0\leq q\leq n$, we conclude the relations
\begin{eqnarray}\label{eq:l2est3}
L^{n,q}_{\gamma}(U) &\subset& L^{n,q}_{\sigma}(U),\\
\label{eq:l2est4}
L^{0,q}_{\sigma}(U) &\subset& L^{0,q}_{\gamma}(U).
\end{eqnarray}

\medskip
Moreover,
for an open set $\Omega \subset X$, $\Omega^*=\Omega \setminus \Sing X$, $\wt{\Omega}:=\pi^{-1}(\Omega)$,
pullback of forms under $\pi$ gives the isometry
\begin{eqnarray}\label{eq:l2est5}
\pi^*: L^2_{p,q}(\Omega^*) \longrightarrow L^{p,q}_{\gamma}(\wt{\Omega}\setminus E) \cong L^{p,q}_{\gamma}(\wt{\Omega}),
\end{eqnarray}
where the last identification is by trivial extension of forms over the thin exceptional set $E$.
Combining \eqref{eq:l2est3} with \eqref{eq:l2est5},
we see that $\pi^*$ maps
\begin{eqnarray}\label{eq:morph1}
\pi^*: L^2_{n,q}(\Omega^*) \longrightarrow L^{n,q}_\sigma(\pi^{-1}(\Omega))
\end{eqnarray}
continuously if $\Omega\subset\subset X$ is a relatively compact open set.
We shall now show how \eqref{eq:morph1} induces the map \eqref{eq:main02a} from 
Theorem \ref{thm:main02}:
\begin{eqnarray}\label{eq:main02ab}
H^{n,q}_{s,loc}(\Omega)
\overset{\cong}{\longrightarrow} H^q\big(\pi^{-1}(\Omega),\mathcal{K}_M\big)\ .
\end{eqnarray}

\medskip
For that, we need a suitable realization of the $L^2_{loc}$-cohomology on $M$.
Let $\mathcal{L}^{p,q}_\sigma$ be the sheaves of germs of forms on $M$ which are locally in $L^{p,q}_\sigma$,
and we denote again by $\dq_w$ the $\dq$-operator in the sense of distributions on such forms
because there is no danger of confusion in what follows.
We can simply use the notation and definitions from Section \ref{sec:dq-complexes} with the choice $X=M$ and $\Sing X=\emptyset$.
Again, we denote the sheaves of germs in the domain of $\dq_w$ by
$$\mathcal{C}^{p,q}_\sigma := \mathcal{L}^{p,q}_\sigma \cap \dq_w^{-1} \mathcal{L}^{p,q+1}_\sigma$$
in the sense that
$\mathcal{C}^{p,q}_\sigma (U) = \mathcal{L}^{p,q}_\sigma (U) \cap \Dom \dq_w(U)$. It is well-known that
$$\mathcal{K}_M := \ker \dq_w \subset \mathcal{C}^{n,0}_\sigma$$
is the usual canonical sheaf on $M$, and that
\begin{eqnarray}\label{eq:res}
0\rightarrow \mathcal{K}_M \hookrightarrow \mathcal{C}^{n,0}_\sigma \overset{\dq_w}{\longrightarrow} \mathcal{C}^{n,1}_\sigma
\overset{\dq_w}{\longrightarrow} \mathcal{C}^{n,2}_\sigma \longrightarrow ...
\end{eqnarray}
is a fine resolution, so that 
$$H^q(U,\mathcal{K}_M) \cong H^q\big(\Gamma(U,\mathcal{C}^{n,*}_\sigma)\big)\ \ ,
\ \ H^q_{cpt}(U,\mathcal{K}_M) \cong H^q(\Gamma_{cpt}\big(U,\mathcal{C}^{n,*}_\sigma)\big)$$
on open sets $U\subset M$.

Now we can use \eqref{eq:morph1} to see that $\pi^*$ defines
a morphism of complexes 
\begin{eqnarray}\label{eq:morph2}
\pi^*: (\mathcal{F}^{n,*},\dq_s) \rightarrow \big(\pi_* (\mathcal{C}^{n,*}_\sigma),\pi_* \dq_w\big).
\end{eqnarray}
Let $\Omega\subset X$ be an open set and let $f\in\mathcal{F}^{n,q}(\Omega)$, $g\in \mathcal{F}^{n,q+1}(\Omega)$
such that $\dq_s f=g$. Then, particularly, also $\dq_w f=g$.
By \eqref{eq:morph1}, it follows that $\pi^* f\in \mathcal{L}^{n,q}_\sigma(\pi^{-1}(\Omega))$
and $\pi^* g\in \mathcal{L}^{n,q+1}_\sigma(\pi^{-1}(\Omega))$ so that $\dq_w \pi^* f=\pi^* g$ on $\pi^{-1}(\Omega)\setminus E$.
But then the $L^2$-extension theorem \cite{Rp1}, Theorem 3.2, tells us that $\dq_w\pi^* f=\pi^*g$ on $\pi^{-1}(\Omega)$.
So $\pi^* f\in \mathcal{C}^{n,q}_\sigma(\pi^{-1}(\Omega))$, $\pi^* g\in \mathcal{C}^{n,q+1}_\sigma(\pi^{-1}(\Omega))$
and \eqref{eq:morph2} is in fact a morphism of complexes.
Including $\mathcal{K}_X^s =\ker\dq_s \subset \mathcal{F}^{n,0}$ and $\mathcal{K}_M=\ker\dq_w \subset \mathcal{C}^{n,0}_\sigma(L)$,
we obtain the commutative diagram
\begin{eqnarray*}
\begin{xy}
  \xymatrix{
      0 \ar[r] & \mathcal{K}_X^s \ar[r] \ar[d]^{\pi^*}    &   \mathcal{F}^{n,0} \ar[r]^{\dq_s} \ar[d]^{\pi^*} & 
\mathcal{F}^{n,1} \ar[r]^{\dq_s} \ar[d]^{\pi^*} & \mathcal{F}^{n,2} \ar[r]
\ar[d]^{\pi^*} & ... \\
      0 \ar[r] & \pi_* \mathcal{K}_M \ar[r] &  \pi_* \mathcal{C}^{n,0}_\sigma \ar[r]^{\pi_* \dq_w}  & 
\pi_* \mathcal{C}^{n,1}_\sigma \ar[r]^{\pi_*\dq_w} & 
\pi_* \mathcal{C}^{n,2}_\sigma \ar[r]
& ... }
\end{xy}
\end{eqnarray*}

It follows from commutativity of the diagram that $\pi^*$ induces morphisms on the cohomology of the complexes,
\begin{eqnarray}\label{eq:morph3}
H^{n,q}_{s,loc} (\Omega) = H^q\big(\Gamma(\Omega,\mathcal{F}^{n,*})\big) 
\overset{\pi^*}{\longrightarrow}
H^q\big(\Gamma(\pi^{-1}(\Omega),\mathcal{C}^{n,*}_\sigma) \big) = H^{n,q}_{w,loc}\big(\pi^{-1}(\Omega)\big),\\
\label{eq:morph4}
H^{n,q}_{s,cpt} (\Omega) = H^q\big(\Gamma_{cpt}(\Omega,\mathcal{F}^{n,*})\big) 
\overset{\pi^*}{\longrightarrow}
H^q\big(\Gamma_{cpt}(\pi^{-1}(\Omega),\mathcal{C}^{n,*}_\sigma) \big) = H^{n,q}_{w,cpt}\big(\pi^{-1}(\Omega)\big),
\end{eqnarray}
for any open set $\Omega\subset X$ and all $q\geq 0$.

\medskip
Note that the pull-back maps $\pi^*$ are continuous on the Level of $L^2$-forms.
Thus, the induced maps are also continuous between the $L^2_{loc}$-spaces if those
carry the natural induced Fr{\'e}chet space topology.
This then descents to the quotient spaces, so that the maps in \eqref{eq:morph3}, \eqref{eq:morph4}
are actually continuous if we consider the cohomology groups as topological quotient spaces.

\medskip
Similarly as above, but a little bit less involved, as one does not need to consider the extension of the $\dq$-equation
across the exceptional set, the push-forward of $L^2$-forms (justified by \eqref{eq:l2est4}),
\begin{eqnarray}\label{eq:morph1e}
\pi_*:  L^{0,q}_\sigma(\pi^{-1}(\Omega)) \rightarrow L^2_{0,q}(\Omega^*)\ ,
\end{eqnarray}
leads to the commutative diagram
\begin{eqnarray*}
\begin{xy}
  \xymatrix{
      0 \ar[r] & \pi_*\OO_M \ar[r] \ar[d]^{\pi_*}    &   \pi_* \mathcal{C}^{0,0}_\sigma \ar[r]^{\pi_* \dq_w} \ar[d]^{\pi_*} & 
\pi_* \mathcal{C}^{0,1}_\sigma \ar[r]^{\pi_* \dq_w} \ar[d]^{\pi_*} & \pi_*\mathcal{C}^{0,2}_\sigma \ar[r]
\ar[d]^{\pi^*} & ... \\
      0 \ar[r] & \ker \dq_w \ar[r] &   \mathcal{C}^{0,0} \ar[r]^{ \dq_w}  & 
 \mathcal{C}^{0,1} \ar[r]^{\dq_w} & 
 \mathcal{C}^{0,2} \ar[r]
& ... }
\end{xy}
\end{eqnarray*}
Again, commutativity of the diagram shows that $\pi_*$ induces morphisms on the cohomology of the complexes,
\begin{eqnarray}\label{eq:morph5}
H^q\big(\Gamma(\pi^{-1}(\Omega),\mathcal{C}^{0,*}_\sigma \big) = H^{0,q}_{w,loc}\big(\pi^{-1}(\Omega)\big)
\overset{\pi_*}{\longrightarrow}
H^{0,q}_{w,loc} (\Omega) = H^q\big(\Gamma(\Omega,\mathcal{C}^{0,*})\big).
\end{eqnarray}
As above, this map $\pi_*$ is continuous if one considers the cohomology groups as topological spaces
with the topology induced from the natural $L^2_{loc}$-Fr{\'e}chet spaces.

\medskip
\subsection{Takegoshi vanishing and first part of Theorem \ref{thm:main02}}

We can now use Takegoshi's vanishing theorem \cite{T} to show that
the induced maps in \eqref{eq:morph3}, \eqref{eq:morph4} are actually isomorphisms.
Let $L \rightarrow M$ be a holomorphic line bundle that is locally semi-positive with respect
to the base space $X$, i.e., assume that $L$ carries a semi-positive Hermitian 
metric on open sets $\pi^{-1}(V)$, where $V$ is a small domain in $X$.

Takegoshi's vanishing theorem (see \cite{T}, Remark 2) tells us that
the higher direct image sheaves of $\mathcal{K}_M(L)$ vanish under these assumptions:
\begin{eqnarray}\label{eq:take1}
R^q\pi_* \big(\mathcal{K}_M(L)\big) =0,\ \ q>0.
\end{eqnarray}
Particularly, $R^q \pi_* \mathcal{K}_M=0$ for $q>0$, so that the complex
\begin{eqnarray*}
\big( \pi_* \mathcal{C}^{n,*}_\sigma, \pi_* \dq_w \big)
\end{eqnarray*}
is exact by the Leray spectral sequence. Here, 
note that the direct image of a fine sheaf under $\pi_*$ is again a fine sheaf.

So, the right-hand side in \eqref{eq:morph3} computes
\begin{eqnarray*}
H^q\big( \pi^{-1}(\Omega), \mathcal{K}_M\big) &\cong& H^q\big( \Omega, \pi_* \mathcal{K}_M \big) = 
H^q\big( \Omega, \mathcal{K}_X \big)\ ,
\end{eqnarray*}
whereas the right-hand side in \eqref{eq:morph4} computes
\begin{eqnarray*}
H^q_{cpt} \big( \pi^{-1}(\Omega), \mathcal{K}_M\big) &\cong& H^q_{cpt}\big( \Omega, \pi_* \mathcal{K}_M \big) = 
H^q_{cpt} \big( \Omega, \mathcal{K}_X \big)\ .
\end{eqnarray*}


On the other hand,
it is shown in \cite{Duke}, Theorem 1.9, that the $\dq_s$-complex
\begin{eqnarray*}
0\rightarrow \mathcal{K}_X^s \hookrightarrow \mathcal{F}^{n,0} \overset{\dq_s}{\longrightarrow}
\mathcal{F}^{n,1} \overset{\dq_s}{\longrightarrow} \mathcal{F}^{n,2} \overset{\dq_s}{\longrightarrow} ... 
\end{eqnarray*}
is a fine resolution of $\mathcal{K}_X^s$ (as $X$ is assumed to have only isolated singularities). It follows that
\begin{eqnarray*}
H^q(\Omega, \mathcal{K}_X^s) &\cong& H^{n,q}_{s,loc}(\Omega) = H^q\big( \Gamma(\Omega,\mathcal{F}^{n,*})\big),\\
H^q_{cpt}(\Omega, \mathcal{K}_X^s) &\cong& H^{n,q}_{s,cpt}(\Omega) = H^q\big( \Gamma_{cpt}(\Omega,\mathcal{F}^{n,*})\big),
\end{eqnarray*}
for the left-hand sides in \eqref{eq:morph3} and \eqref{eq:morph4}, respectively.

\bigskip
Assume now that $\Omega$ contains only singularities of the first kind \eqref{eq:type0} in Theorem \ref{thm:main01}.
Then $\mathcal{K}_X^s = \mathcal{K}_X$ on $\Omega$, and so we see that the induced mappings $\pi^*$
in \eqref{eq:morph3} and \eqref{eq:morph4}, respectively, are actually isomorphisms. This, finally, proves the first statement
of Theorem \ref{thm:main02}.

\medskip
\subsection{$L^2$-Serre duality and second part of Theorem \ref{thm:main02}}

Consider the diagram
\begin{eqnarray}\label{eq:diagram100}
\begin{xy}
  \xymatrix{\mathcal{L}^{0,q}(\Omega) & \times & 
  \mathcal{L}^{n,n-q}_{cpt}(\Omega) \ar[d]^{\pi^*} \ar[r]^{\ \ \ \ \ \int} & \C \\
  \mathcal{L}^{0,q}_\sigma\big(\pi^{-1}(\Omega)\big) \ar[u]^{\pi_*} & \times & \mathcal{L}^{n,n-q}_{cpt}\big(\pi^{-1}(\Omega)\big) \ar[r]^{\ \ \ \ \ \ \ \  \int} & \C\ ,}
  \end{xy}
\end{eqnarray}
where both pairings are given by integration $(a,b) \mapsto \int a\wedge b$. The integrals exist by the H\"older inequality
as we are integrating products of $L^2_{loc}$-forms where $b$ has compact support.
By the transformation law, we have
\begin{eqnarray*}
\int_{\Omega^*} \pi_* \eta \wedge \omega &=& \int_{\pi^{-1}(\Omega)} \eta \wedge \pi^* \omega\ ,
\end{eqnarray*}
where, as above, $\Omega^*=\Omega \setminus \Sing X$,
and we have replaced $\pi^{-1}(\Omega) \setminus E$ on the right hand side
by $\pi^{-1}(\Omega)$ as the exceptional set $E$ is neglectable. 

\bigskip
Classically, 
the lower line in the diagram \eqref{eq:diagram100} induces a pairing on the level of $\dq$-cohomology,
as can be seen by use of partial integration.
The same holds true for the upper line if we consider a pairing between $\dq_w$- and $\dq_s$-cohomology.
The reason is that partial integration is possible even on $X$ if one of the forms is in the domain of $\dq_s$
so that it can be approximated in the graph norm by forms with support away from the singular set. Then, partial integration
is possible across $\Sing X$ (cf. \cite{Serre} for more details). Hence, we obtain the induced diagram
\begin{eqnarray}\label{eq:diagram101}
\begin{xy}
  \xymatrix{ H^{0,q}_{w,loc}(\Omega) & \times & 
  H^{n,n-q}_{s,cpt}(\Omega) \ar[d]^{\pi^*} \ar[r]^{\ \ \ \ \ \int} & \C \\
  H^{0,q}_{w,loc} \big(\pi^{-1}(\Omega)\big) \ar[u]^{\pi_*} & \times & H^{n,n-q}_{w,cpt}\big(\pi^{-1}(\Omega)\big) \ar[r]^{\ \ \ \ \ \ \ \  \int} & \C\ ,}
  \end{xy}
\end{eqnarray}
where the vertical arrows are just the mappings $\pi_*$ from \eqref{eq:morph5} (left-hand side) 
and $\pi^*$ from \eqref{eq:morph4} (right-hand side). We have now
\begin{eqnarray}\label{eq:pairing}
\int_{\Omega^*} \pi_* [\eta] \wedge [\omega] &=& \int_{\pi^{-1}(\Omega)} [\eta] \wedge \pi^* [\omega]
\end{eqnarray}
for cohomology classes.

\bigskip
We have already seen that the pull-back $\pi^*$ in \eqref{eq:diagram101} is an isomorphism.
Now, we clearly intend to use Serre-duality to conclude that $\pi_*$ is also an isomorphism.
For this, we require the Hausdorff property for the cohomology groups under consideration,
considered as topological spaces. 

To achieve that, let us first assume that $\Omega$ is just a small neighborhood of an isolated singularity $x_0\in X$
with smooth strongly pseudoconvex boundary $b\Omega$.
Consider for example $X \cap B_\epsilon(x_0)$, where $B_\epsilon(x_0)$ is some small ball
in the ambient space for a local embedding.
Then $\pi^{-1}(\Omega)$ also has smooth strongly pseudoconvex boundary,
and so $H^{0,q}_{w,loc}\big(\pi^{-1}(\Omega)\big)$ is of finite dimension for all $q>0$
by classical $\dq$-theory on complex manifolds (see \cite{G1}).
Hence, $H^{0,q}_{w,loc}\big(\pi^{-1}(\Omega)\big)$ is particularly Hausdorff for all $q\geq 0$
(the trivial case $q=0$ can be clearly included),
and so
the pairing in the lower line of \eqref{eq:diagram101}
is non-degenerate by classical Serre duality (see \cite{JPSerre}, Th{\'e}or{\`e}me 2).
So, we conclude that
\begin{eqnarray*}
H^{n,n-q}_{w,cpt}\big( \pi^{-1}(\Omega) \big) &\cong& H^{n,n-q}_{s,cpt}(\Omega)
\end{eqnarray*}
is of finite dimension for $1\leq q \leq n$, and Hausdorff for $q=0$.
The last claim follows from the fact that
the map $\pi^*$ on the right hand side of \eqref{eq:diagram101} is a topological isomorphism
(cf. also the topological considerations in \cite{Serre}).

But, by the $L^2$-duality on singular complex spaces introduced in \cite{Serre} (Lemma 3.5),
for $0\leq q \leq n-1$,
$H^{n,n-q}_{s,cpt}(\Omega)$ is Hausdorff if and only if $H^{0,q+1}_{w,loc}(\Omega)$
is Hausdorff. So we can conclude by \cite{Serre}, Theorem 1.4, that 
the upper line of \eqref{eq:diagram101} is also non-degenerate ($H^{0,0}_{w,loc}(\Omega)$ is trivially Hausdorff).
So, $\pi_*$ is also an isomorphism in \eqref{eq:diagram101} for all $0\leq q\leq n$.
The case $q=0$ follows also directly from normality of $X$.

\medskip
We will now remove the assumption that the boundary of $\Omega$ is strongly pseudoconvex
by use of a Mayer-Vietoris sequence.
Let $\Omega$ be a domain in $X$ and $\pi: M\rightarrow \Omega$ the resolution of the isolated singularities
inside $\Omega$. We do not resolve possible singularities in the boundary of $\Omega$,
so that the metric is not modified at the boundary of $\Omega$.
The set of singularities is clearly discrete in $\Omega$.

At each singularity $x_\nu\in\Omega$ choose a small ball
$$B_\nu \ := \ X\cap B_{2\epsilon}(x_\nu)$$
such that $B_\nu \subset\subset \Omega$ and the $B_\nu$ are pairwise disjount.
Let
$$D_\nu \ := \ X\cap B_{\epsilon}(x_\nu).$$
We will apply the Mayer-Vietoris sequence for Dolbeault cohomology to the open sets
$$U \ := \ \bigcup_{\nu\in\Sing \Omega} B_\nu \ \ \ , \ \ \ V \ := \ \Omega \setminus \bigcup_{\nu\in\Sing \Omega} \overline{D_\nu}.$$

We can assume that the metric is not modified under the resolution $\pi: M\rightarrow \Omega$
on a neighborhood of $V$. Hence, for all $q\geq 0$, the maps 
\begin{eqnarray*}
\pi_*: H^{0,q}_{w,loc} \big(\pi^{-1}(V)\big) &\longrightarrow& H^{0,q}_{w,loc}(V)\\
\pi_*: H^{0,q}_{w,loc} \big(\pi^{-1}(U\cap V)\big) &\longrightarrow& H^{0,q}_{w,loc}(U \cap V)
\end{eqnarray*}
are trivially isomorphic, whereas we obtain the isomorphisms
$$\pi_*: H^{0,q}_{w,loc} \big(\pi^{-1}(U)\big) \longrightarrow H^{0,q}_{w,loc}(U)$$
by what we have seen above ($U$ is a disjoint union of domains with strongly pseudoconvex boundary).

Let $\wt{U}:=\pi^{-1}(U)$, $\wt{V}:=\pi^{-1}(V)$, $W:=U\cap V$, $W:=\pi^{-1}(U\cap V)$.
 Consider now the commutative diagram
\begin{eqnarray*}
\begin{xy}
  \xymatrix{
      H^{q-1}(U) \oplus H^{q-1}(V) \ar[r]     &   H^{q-1}(W) \ar[r]  & 
H^{q}(\Omega) \ar[r] & H^{q}(U) \oplus H^{q}(V) \ar[r]
 & H^{q}(W)\\
H^{q-1}(\wt{U}) \oplus H^{q-1}(\wt{V}) \ar[r] \ar[u]^{\pi_*}_{\cong}   &   H^{q-1}(\wt{W}) \ar[r] \ar[u]^{\pi_*}_{\cong}& 
H^{q}\big(\pi^{-1}(\Omega)\big) \ar[r]  \ar[u]^{\pi_*} & H^{q}(\wt{U}) \oplus H^{q}(\wt{V}) \ar[r] \ar[u]^{\pi_*}_{\cong} 
 & H^{q}(\wt{W})  \ar[u]^{\pi_*}_{\cong} \ ,}
\end{xy}
\end{eqnarray*}
where we write $H^q$ in place of $H^{0,q}_{w,loc}$ for ease of notation.
Both, the upper and the lower line of the diagram are exact by the Mayer-Vietoris sequence for Dolbeault cohomology,
which is valid because the sheaves of $\mathcal{L}^2$-forms
are all fine, i.e., admit partition of unities.

The two vertical arrows on the left hand side and the two vertical arrows on the right hand side
are ismorphisms by what we have seen above.

Hence, push-forward of forms induces for any $q\geq0$ a natural isomorphism
\begin{eqnarray*}
H^q\big(\pi^{-1}(\Omega),\OO_M\big) 
\overset{\cong}{\longrightarrow} H^{0,q}_{w,loc}(\Omega).
\end{eqnarray*}
by use of the $5$-Lemma in the commutative diagram. This proves the second statement of Theorem \ref{thm:main02}.

\medskip
\subsection{Singularities of the second kind}

We will now prove the third statement, \eqref{eq:main02c}, of Theorem \ref{thm:main02}.
To do so, we use again Takegoshi's vanishing theorem \cite{T},
similar to the proof of the first statement of our theorem.

\medskip
As $\Omega$ contains only singularities of the second kind \eqref{eq:type1} in Theorem \ref{thm:main01},
it follows that
$$\mathcal{K}_X^s \ \ = \ \ \pi_* \big( \mathcal{K}_M \otimes \OO(-Z)\big)$$
over $\Omega$, where $\pi: M \rightarrow X$ is a resolution of singularities
and $Z:=\pi^{-1}(\Sing X)$ the unreduced exceptional divisor.
The proof of \eqref{eq:main02c} follows now completely analogous 
to the first statement of the theorem by use of the Leray spectral sequence
if we can show that
\begin{eqnarray}\label{eq:take77}
R^q \pi_* \big( \mathcal{K}_M \otimes \OO(-Z)\big) \ \ = \ \ 0
\end{eqnarray}
for $q>0$. But this is again a consequence of Takegoshi's vanishing theorem
because $\OO(-Z)$ is locally semi-positive with respect to $\pi$,
as we will explain in the following.

\bigskip
We can use Takegoshi's vanishing theorem \cite{T} to show that
the induced maps in \eqref{eq:morph3}, \eqref{eq:morph4} are actually isomorphisms.
Let $L \rightarrow M$ be a holomorphic line bundle that is locally semi-positive with respect
to the base space $X$, i.e., assume that $L$ carries a semi-positive Hermitian 
metric on open sets of the form $\pi^{-1}(V)$ where $V$ is a small domain in $X$.

Takegoshi's vanishing theorem (see \cite{T}, Remark 2) tells us that
the higher direct image sheaves of $\mathcal{K}_M(L)$ vanish under these assumptions:
\begin{eqnarray*}
R^q\pi_* \big(\mathcal{K}_M(L)\big) =0,\ \ q>0.
\end{eqnarray*}

\medskip
Let $(M,\sigma)$ be a Hermitian complex manifold and $D$ a divisor on $M$.
Let $\OO(D)$ be the sheaf of germs of meromorphic functions $f$ such that $\mbox{div}(f)+D\geq 0$.
We denote by $L_D$ the associated holomorphic line bundle such that sections in $\OO(D)$ correspond
to holomorphic sections in $L_D$.
The constant function $f\equiv 1$ induces a meromorphic section $s_D$ of $L_D$
such that $\mbox{div}(s_D)=D$. One can then identify sections in $\OO(Z)$
with sections in $\OO(L_D)$ by $g\mapsto g\otimes s_D$,
and we denote the inverse mapping by $u\mapsto u\cdot s_D^{-1}$.
If $D$ is an effective divisor, then $s_D$ is a holomorphic section of $L_D$ and
$\OO(-D) \subset \OO \subset \OO(D)$.

More generally, if $F$ is an arbitrary divisor and $D$ is effective, then there is the natural inclusion $\OO(F)\subset \OO(F+D)$
which induces the inclusion $\OO(L_{F})\subset \OO(L_{F+D})$ given by
$u\mapsto (u\cdot s_{F}^{-1})\otimes s_{F+D}$.
For open sets $U\subset M$, this also induces the natural inclusion of smooth sections of vector bundles 
\begin{eqnarray*}
\Gamma(U,L_{F}) \subset \Gamma(U,L_{F+D}).
\end{eqnarray*}

\medskip
Let now $Z:=\pi^{-1}(\Sing X)$ be the unreduced exceptional divisor,
and $L_{-Z}$ the line bundle associated to $\OO(-Z)$,
the sheaf of holomorphic functions vanishing at least to the order of $Z$ on the exceptional set.
We will now show that $L_{-Z}$ is locally semi-positive with respect to $\pi$.
Note that a metric on $L_{-Z}$ is given by a non-vanishing smooth section of 
$L_{-Z}^*\otimes\o{L^*_{-Z}} = L_Z \otimes \o{L_Z}$.

So, let $x_0\in X$ be an isolated singularity. We can assume that $X$ is locally embedded in some $\C^N$
such that $x_0=0 \in \C^N$. Let $z_1, ..., z_N$ be the standard Euclidean coordinates of $\C^n$.
Then $Z$ is the common zero set of the holomorphic functions $\pi^* z_1, ... , \pi^* z_N$,
$$Z \ = \ V \big( \pi^* z_1, ... , \pi^* z_N \big).$$
So, the $\pi^* z_1, ..., \pi^* z_N$ generate $\OO(-Z)$. They define holomorphic sections
$s_j \ := \ \pi^* z_j \otimes s_{-Z}$ of $L_{-Z}$,
which have no common zeroes. Let
$$ F \ \ := \ \  |\pi^* z_1|^2 + ... + |\pi^* z_N|^2.$$
Then
$$ h:= F^{-1} \otimes s_Z \otimes \o{s_Z}$$
is a non-vanishing smooth section of $L_Z \otimes \o{L_Z}$,
hence defining a smooth Hermitian metric
$$h \ = \ e^{-\varphi}$$
on $L_{-Z}$, where $\varphi$ is a plurisubharmonic weight: let $g$ be a holomorphic function with $\mbox{div}(g)=Z$.
Then, in a local frame for $L_Z \otimes \o{L_Z}$,
$$ h \ \ = \ \ \frac{|g|^2}{ |\pi^* z_1|^2 + ... + |\pi^* z_N|^2},$$
so that
$$\varphi \ \ = \ \ \log \big( |\pi^* z_1/g|^2 + ... + |\pi^* z_N/g|^2\big).$$

Hence, $L_{-Z}$ is actually locally semi-positive with respect to $\pi$,
and hence, as desired,
$$R^q\pi_* \big( \mathcal{K}_M ( L_{-Z} ) \big) \ = \ R^q\pi_* \big( \mathcal{K}_M \otimes \OO(-Z) \big) \ = \ 0$$
for $q>0$ by Takegoshi's vanishing theorem.

\bigskip
Finally, by use of Serre duality between $H^{0,q}_{w,loc} \big( \pi^{-1}(\Omega), L_Z\big)$
and $H^{n,n-q}_{w,cpt}\big( \pi^{-1}(\Omega), L_{-Z}\big)$,
the last statement, \eqref{eq:main02d}, of Theorem \ref{thm:main02}
follows completely analogous to the second statement, \eqref{eq:main02b}, as above.

\newpage

\section{A vanishing theorem of Kodaira--Grauert--Riemenschneider--Takegoshi type}\label{sec:vanishing}

In this short section, we shall prove Theorem \ref{thm:main03}.
Recall the following:

\begin{thm}\label{thm:NK}
Let $M$ be a compact K\"ahler manifold of dimension $n$ and $E$ an almost positive
line bundle on $M$. Then: $H^{n,q}(M,E)=0$ for $q>0$.
\end{thm}

A proof, which is standard complex differential geometry by means of the Bochner--Kodaira--Nakano
inequality can be found, e.g., in \cite{Duke}, Theorem 3.6.

\medskip
It is now straight forward to prove Theorem \ref{thm:main03} by means of Theorem \ref{thm:main01}
and Theorem \ref{thm:main02}.
So, let $X$ be an irreducible subvariety of a compact K\"ahler manifold with only isolated
canonical Gorenstein singularities of the first kind in Theorem \ref{thm:main01},
and $L\rightarrow X$ an almost positive holomorphic line bundle over $X$.

We may assume that $\pi: M\rightarrow X$ is an embedded resolution
obtained by finitely many blow-ups (i.e. monoidal transformations) along smooth centers, see \cite{BM}, Theorem 13.4.
So, $M$ can be interpreted as a submanifold in a finite product of K\"ahler manifolds,
inheriting a K\"ahler metric.


It is now easy to see just as in the proof of Theorem \ref{thm:main02} by means
of Takegoshi vanishing, that
pull-back of forms induces for any $q\geq0$ a natural isomorphism
\begin{eqnarray}\label{eq:main02a'}
H^q\big(X, \mathcal{K}_X^s (L) \big) \ \cong \ H^{n,q}_s(X,L)
\overset{\cong}{\longrightarrow} H^q\big(M,\mathcal{K}_M (\pi^*L)\big).
\end{eqnarray}
But $\pi^* L$ remains almost positive. 
So, the cohomology groups in \eqref{eq:main02a'} vanish for $q>0$ by Theorem \ref{thm:NK}.
By $L^2$-Serre duality for singular spaces, \cite{Serre},
we deduce that also 
$$H^{0,n-q}_w(X,L^*) \ = \ 0$$
for all $q>0$.




\bigskip
\section{Characterization of canonical surface singularities}



In this section, we prove Theorem \ref{thm:main04}, i.e., canonical surface singularities are classified
with respect to the distinction in Theorem \ref{thm:main01}.

Let $x_0 \in X$ be a so-called du~Val singularity, i.e., a canonical singularity of dimension two.
These can be realized as hypersurface singularities in $\C^3$. So, $x_0$ is automatically Gorenstein.

\medskip
\subsection{$A_n$-singularities}


In this subsection, we prove that $A_n$-singularities, $n\geq 1$, are of the first kind \eqref{eq:type0} in Theorem \ref{thm:main01}.

\medskip
We can assume without loss of generality that the $A_n$-singularity is realised as the origin
in
$$X \ \ = \ \ \{ z^{n+1}-xy=0 \} \ \ \subset \C^3\ ,$$
where the hypersurface $X$ is given as the zero set of $f(x,y,z) = z^{n+1} -xy$.
Let us deal with $X$ as the quotient singularity described by the
branched $n+1$ to $1$ covering
\begin{eqnarray*}
\pi: \C^2 &\longrightarrow & X\ ,\\
(s,t) &\mapsto& (s^{n+1}, t^{n+1}, st )\ .
\end{eqnarray*}

\bigskip
To prove that the $A_n$-singularity is of the first kind, we will compute explicitly
that the generator of $\mathcal{K}_{X}$, the so-called {\it structure form} or {\it Poincar\'e residue},
is in the domain of $\dq_s$. This shows that $\mathcal{K}_{X} \subset \mathcal{K}_{X^s}$,
and the other inclusion is clear by definition.

The singularity of $X$ is a canonical Gorenstein singularity. So, $\mathcal{K}_{X} = \omega_{X}$
is locally free of rank one. By use of the analytic realisation of the adjunction formula (see, e.g., \cite{RSW}),
it is generated by the Poincar\'e residue
\begin{eqnarray*}
\omega \ \ = \ \ \frac{dx\wedge dy}{\partial f / \partial z} \ \ = \ \ \frac{dx\wedge dy}{(n+1)  z^n}\ .
\end{eqnarray*}
We need to show that $\omega$ is in the domain of the $\dq_s$-operator.
This can be done by use of the cut-off procedure described in Section \ref{ssec:cut-off}.

\bigskip
Hence, let
\begin{eqnarray*}
\omega_k (\zeta) \ \ := \ \ \mu_k(\zeta) \cdot \omega(\zeta)\ ,
\end{eqnarray*}
where $\mu_k(\zeta):=\rho_k\big(\log(-\log r(\|\zeta\|))\big)$ is the cut-off function \eqref{eq:cutoff1}
from Section \ref{ssec:cut-off}.
As above, since the $\mu_k$ are uniformly bounded in $k$,
it follows by Lebesgue's dominated convergence that
\begin{eqnarray*}
\omega_k = \mu_k \cdot \omega \ \longrightarrow \omega
\end{eqnarray*}
in $L^{2,0}$. It remains to prove that also
\begin{eqnarray*}
\dq \omega_k = \dq \mu_k \wedge \omega \ \longrightarrow 0
\end{eqnarray*}
in $L^{2,1}$. This follows again by dominated convergence once we can show that the
$L^2$-norm of the $\dq \omega_k$ is also uniformly bounded in $k$.
To see this, recall that there is a
constant $C>0$ such that
\begin{equation}\label{eq:cutoff2c}
\big| \dq \mu_k(\zeta)\big|_{X} \ \leq \ \big| \dq \mu_k(\zeta)\big|_{\C^3} \ \leq \ 
C\frac{\chi_k(|\zeta|)}{\|\zeta\| \big| \log\|\zeta\|\big|},
\end{equation}
where $\chi_k$ is the characteristic function of $[e^{-e^{k+1}}, e^{-e^k}]$.
Here, note that the norm of the form $\dq\mu_k$
is only decreasing when pulled-back to $X$.

So, we need an estimate for
\begin{eqnarray*}
I_k &:=& \|\dq\mu_k \wedge \omega\|^2_{X} \ \ = \ \ \int_{X} |\dq \mu_k|^2_{X} |\omega|^2_{X} dV_{X} \\
&=& \int_{D_k} |\dq \mu_k|^2_{X} \ \omega \wedge \o{\omega} \ \ \leq \ \
C \int_{D_k} \frac{\omega \wedge \o{\omega}}{\|\zeta\|^2 \log^2 \|\zeta\|} \ ,
\end{eqnarray*}
where we let $D_k = \{\zeta\in X: e^{-e^{k+1}} < |\zeta| < e^{-e^k} \}$
and use $|\omega|^2_{X} dV_{X} = \omega \wedge \o{\omega}$.

\bigskip
Integrating directly on $X$ is tedious, but we can use the covering $\pi$ to pull-back
the problem to $\C^2$ with the convenient fact that
\begin{eqnarray*}
\pi^* \omega\ (s,t) &=& \frac{ds^{n+1} \wedge dt^{n+1}}{(n+1) (st)^n} \ = \ (n+1)\ ds\wedge dt\ .
\end{eqnarray*}
Hence the integral $I_k$ can be estimated from above by
\begin{eqnarray*}
\wt{I_k} &:=& \int_{\pi^{-1}(D_k)} \frac{dV(s,t)}{\big( |s|^{2n+2} + |t|^{2n+2} + |st|^2\big) \log^2\big( |s|^{2n+2} + |t|^{2n+2} + |st|^2\big)}.
\end{eqnarray*}
But these integrals are actually uniformly bounded, independent of $k$,
as was calculated in the proof of \cite{Slarw}, Theorem 5.1, see equation (5.3) ibid.

\bigskip
\subsection{Non $A_n$-singularities}

In this subsection, we prove by contradiction that singularities of type $D_n$, $n\geq 4$, 
and of type $E_6$, $E_7$ or $E_8$ 
are of the second kind \eqref{eq:type1} in Theorem \ref{thm:main01}.
So, let $x_0$ be one of these, and
assume that $x_0$ is of the first kind \eqref{eq:type0} in Theorem \ref{thm:main01}

\medskip
Let $\pi: M \rightarrow X$ be a resolution of singularities. Then
$$\big(R^1 \pi_* \OO_M\big)_{x_0} \ = \ 0\ ,$$
because $x_0$ is a rational singularity.
This implies by the use of Theorem \ref{thm:main02}, \eqref{eq:main02b}, that the $L^2$-$\dq_w$-complex
\begin{eqnarray*}\label{eq:intro2b}
0\rightarrow \mathcal{O}_X \hookrightarrow \mathcal{C}^{0,0} \overset{\dq_w}{\longrightarrow}
\mathcal{C}^{0,1} \overset{\dq_w}{\longrightarrow} \mathcal{C}^{0,2} \overset{\dq_w}{\longrightarrow} ... 
\end{eqnarray*}
is particularly exact at $(\mathcal{C}^{0,1})_{x_0}$,
as we assume that $x_0$ is of the first kind. This means that the $\dq_w$-equation is locally solvable in $x_0$
on the level of $(0,1)$-forms. But that contradicts \cite{Slarw}, Theorem 5.6,
where we have computed that there are obstructions to this solvability for the types of singularity
under observation. 

Hence, by contradiction, $x_0$ must be of the second kind.


\bigskip
\begin{thebibliography}{99999}





\bibitem[A]{Alt} {\sc H.\ W.\ Alt}, {\em Lineare Funktionalanalysis}, Springer-Verlag, Berlin, 1992.


\bibitem[ALRSW]{Slarw} {\sc M.\ Anderson, R.\ L\"ark\"ang, J.\ Ruppenthal, H.\ Samuelsson\ Kalm, E.\ Wulcan},
Estimates for the $\dq$-equation on canonical surfaces,
{\em J. Geom. Anal.} {\bf 30} (2020), 2974--3001.



\bibitem[BM]{BM} {\sc E.\ Bierstone, P.\ Milman},
Canonical desingularization in characteristic zero by blowing-up the maximum strata of a local invariant,
{\em Inventiones Math.} {\bf 128} (1997), {\em no. 2}, 207--302.



\bibitem[D1]{D}{\sc J.-P.\ Demailly}, {\em Complex Analytic and Differential Geometry},
online book, available at {\sf www-fourier.ujf-grenoble.fr/$\sim$demailly/manuscripts/agbook.pdf}, Institut Fourier, Grenoble.




\bibitem[DS]{DiSi}{\sc T.-C.\ Dinh, N.\ Sibony},
Pull-back of currents by holomorphic maps, {\em manuscripta math.} {\bf 123} (2007), 357--371.


\bibitem[D2]{Du}{\sc A.\ H.\ Durfee}, Fifteen characterizations of rational double points and simple critical points,
{\em Enseign. Math. (2)} {\bf 25} (1979), 131--163.










\bibitem[G1]{G1}{\sc H.\ Grauert},
On Levi's problem and the embedding of real analytic manifolds,
{\em Ann. Math.} {\bf 68} (1958), 460--472.


\bibitem[G2]{G2} {\sc H.\ Grauert},
\"Uber Modifikationen und exzeptionelle analytische Mengen,
{\em Math. Ann.} {\bf 146} (1962), 331--368.




\bibitem[GR]{GR}{\sc H.\ Grauert, O.\ Riemenschneider},
Verschwindungss\"atze f\"ur analytische Kohomologiegruppen auf komplexen R\"aumen,
{\em Invent. Math.} {\bf 11} (1970), 263--292.





\bibitem[K]{Ka}{\sc U.\ Karras}, Local cohomology along exceptional sets,
{\em Math. Ann.} {\bf 275} (1986), 673--682.




\bibitem[L]{La}{\sc H.\ Laufer}, On rational singularities,
{\em Amer. J. Math.} {\bf 94} (1972), 597--608.




\bibitem[OV]{OV}{\sc N.\ {\O}vrelid, S.\ Vassiliadou},
$L^2$-$\dq$-cohomology groups of some singular complex spaces,
{\em Invent. Math.} {\bf 192} (2013), no. 2, 413--458.




\bibitem[P1]{P1}{\sc W. Pardon}, The $L^2$-$\dq$-cohomology of an algebraic surface,
{\em Topology} {\bf 28} (1989), no. 2, 171--195.


\bibitem[PS1]{PS1}{\sc W.\ Pardon, M.\ Stern},
$L^2$-$\dq$-cohomology of complex projective varieties, {\em J. Amer. Math. Soc.} {\bf 4} (1991), no. 3, 603--621.



\bibitem[PS2]{PS2}{\sc W.\ Pardon, M.\ Stern},
Pure Hodge structure on the $L^2$-cohomology of varieties with isolated singularities,
{\em J. reine angew. Math.} {\bf 533} (2001), 55--80.






\bibitem[P2]{P} {\sc D.\ Prill}, The divisor class groups of some rings of holomorphic functions,
{\em Math. Z.} {\bf 121} (1971), 58--80.



\bibitem[RR]{RR}{\sc J.-P.\ Ramis, G.\ Ruget}, Complexe dualisant et th\'eor\`eme de dualit\'e en g\'eom\'etrie analytique complexe,
{\em Inst. Hautes \'Etudes Sci. Publ. Math.} {\bf 38} (1970), 77--91.




\bibitem[R1]{Rd}{\sc W.\ Rudin}, {\em Functional Analysis}, 
International Series in Pure and Applied Mathematics, McGraw-Hill, New York, 1991.


\bibitem[R2]{Rp1}{\sc J.\ Ruppenthal}, About the $\dq$-equation at isolated singularities with regular exceptional set,
{\em Internat. J. Math.} {\bf 20} (2009), no. 4, 459--489.




\bibitem[R3]{Rp7}{\sc J.\ Ruppenthal}, The $\dq$-equation on homogeneous varieties with an isolated singularity,
{\em Math. Z.} {\bf 263} (2009), 447--472.





\bibitem[R4]{Duke} {\sc J.\ Ruppenthal}, $L^2$-theory for the $\dq$-operator on compact complex spaces,
{\em Duke Math. J.} {\bf 163} (2014), no. 15, 2887--2934.





\bibitem[R5]{parabolic}{\sc J.\ Ruppenthal},
Parabolicity of the regular locus of complex varieties,
{\em Proc. Amer. Math. Soc.} {\bf 144} (2016), no.1, 225--233



\bibitem[R6]{Serre}{\sc J.\ Ruppenthal},
$L^2$-Serre duality on singular complex spaces and rational singularities,
{\em Int. Math. Res. Not. IMRN} {\bf 23} (2018), 7198--7240.


\bibitem[R7]{Toulouse}{\sc J.\ Ruppenthal}.
$L^2$-theory for the $\dq$-operator on complex spaces with isolated singularities,
{\em Ann. Fac. Sci. Toulouse Math.} {\bf 28} (2019), no.2, 225--258.


\bibitem[RSW]{RSW}{\sc J.\ Ruppenthal, H.\ Samuelsson Kalm, E.\ Wulcan}.
Adjunction for the Grauert--Riemenschneider canonical sheaf and extension of  $L^2$-cohomology classes,
{\em Indiana Univ. Math. J.} {\bf 64} (2015), no. 2, 533--558.






\bibitem[S1]{JPSerre}{\sc J.-P.\ Serre},
Un th{\'e}or{\`e}me de dualit{\'e}, {\em Comm. Math. Helv.} {\bf 29} (1955), 9--26.



\bibitem[S2]{Sibony}{\sc N.\ Sibony},
Quelques probl\`emes de prolongement de courants en analyse complexe,
{\em Duke Math. J.} {\bf 52} (1985), no. 1, 157--197.






\bibitem[T]{T} {\sc K.\ Takegoshi}, Relative vanishing theorems in analytic spaces,
{\em Duke Math. J.} {\bf 51} (1985), no. 1, 273--279.



\end{thebibliography}
\end{document}